\documentclass[11pt]{amsart}
\usepackage{pdfsync}
\usepackage{amssymb,amsthm,amsmath}
\usepackage{enumerate}
\usepackage[bookmarks,colorlinks,breaklinks]{hyperref}
\usepackage{color}
\newtheorem{thm}{Theorem}

\newtheorem{prop}[thm]{Proposition}
\newtheorem{lem}[thm]{Lemma}

\newtheorem{question}[thm]{Question}

\theoremstyle{definition}
\newtheorem{defn}[thm]{Definition}

\theoremstyle{definition}

\newtheorem{rem}[thm]{Remark}

\numberwithin{thm}{section}
\numberwithin{equation}{thm}

\newcommand{\bP}{\ensuremath{{\mathbb{P}}}}

\newcommand{\bZ}{\ensuremath{{\mathbb{Z}}}}
\renewcommand{\P}{\ensuremath{{\mathbb{P}}}}

\newcommand{\N}{\ensuremath{{\mathbb{N}}}}
\newcommand{\p}{\ensuremath{{\mathfrak{p}}}}

\newcommand{\fp}{\ensuremath{{\mathfrak{p}}}}
\newcommand{\fq}{\ensuremath{{\mathfrak{q}}}}
\newcommand{\fm}{\ensuremath{{\mathfrak{m}}}}
\newcommand{\fo}{\ensuremath{{\mathfrak{o}}}}

\newcommand{\cS}{\mathcal S}

\newcommand{\cM}{\mathcal M}

\newcommand{\lra}{\longrightarrow}
\newcommand{\Kbar}{\ensuremath {\overline K}}

\newcommand{\ba}{{\boldsymbol \alpha}}

\DeclareMathOperator{\Gal}{Gal}

\DeclareMathOperator{\Aut}{Aut}

\begin{document}

\title[Iterated Galois groups of unicritical polynomials]{Finite index
  theorems for iterated Galois groups of preperiodic points
for unicritical polynomials}



\author[Minsik Han]{Minsik Han}
\address{Minsik Han\\Department of Mathematics\\ University of Rochester\\
Rochester, NY, 14620, USA}
\email{minsik.han@rochester.edu}

\author[T. J. Tucker]{Thomas J. Tucker}
\address{Thomas J. Tucker\\Department of Mathematics\\ University of Rochester\\
Rochester, NY, 14620, USA}
\email{tjtucker@gmail.com}



\date{}

\dedicatory{}

\begin{abstract}
  Let $K$ be a number field. Let $f\in K[x]$ be of the form
  $f(x)=x^q+c$, where $q$ is a prime power.  Let $\beta\in K$. For
  all $n\in\mathbb{N}\cup\{\infty\}$, the Galois groups
  $G_n(\beta)=\Gal(K(f^{-n}(\beta))/K(\beta))$ embed into a group
  ${\bf G}_n$ containing  $[C_q]^n$, the $n$-fold
  wreath product of the cyclic group $C_q$, as a subgroup of index
  $[K(\xi_q): K]$ for $\xi_q$ a primitive $q$-th root of unity.  Let
  $G_\infty(\beta)$ and ${\bf G}_\infty $ be the inverse limits of
  $G_n (\beta)$ and ${\bf G}_n$ respectively.  

  We show that if
  $f$ is not post-critically finite and
  $\beta$ is strictly preperiodic under
  $f$, then $[{\bf G}_\infty:G_\infty(\beta)]<\infty$.
\end{abstract}

\maketitle


\section{Introduction and Statement of Results}\label{intro}
Let $K$ be a field.  Let $f\in K(x)$ with $d=\deg f\geq 2$ and let
$\beta\in K$.  For $n\in\N$, let
$K_n(\beta)=K(f^{-n}(\beta))$ denote the field obtained by adjoining the
$n$th preimages of $\beta$ under $f$ to $K(\beta)$. (We declare that
$K_0(\beta)=K$.)  Set
$K_\infty(\beta)=\bigcup_{n=1}^\infty K_n(\beta)$.  For
$n\in\N\cup\{\infty\}$, define
$G_n(\beta)=\Gal(K_n(\beta)/K(\beta))$, and let $G_\infty(\beta) =
\varprojlim G_n (\beta)$ (note that there are natural projection maps $G_i (\beta)
\lra G_j (\beta)$ for $i  > j$).  

When $\beta$ is not in the forward orbit of any critical point.  the
group $G_\infty(\beta)$ embeds into $\Aut(T^d_\infty)$, the
automorphism group of an infinite $d$-ary rooted tree
$T^d_\infty$. Recently there has been much work on the problem of
determining when the index $[\Aut(T^d_\infty):G_\infty(\beta)]$ is
finite. The group $G_\infty(\beta)$ is the image of an arboreal Galois
representation, so this finite index problem is an analog in
arithmetic dynamics of the finite index problem for the $\ell$-adic
Galois representations associated to elliptic curves, resolved by
Serre's celebrated Open Image Theorem~\cite{Serre}. By work of
Odoni~\cite{OdoniIterates}, one expects that many rational functions
have a surjective arboreal representation, i.e., that
$[\Aut(T^d_\infty):G_\infty(\beta)]=1$ (see also \cite{OC, NL,
  Kadets, False, RafeS}).  

For special rational functions (such as $f(x) = x^d + c$ for
$d \geq 3$), $G_\infty(\beta)$ will be much smaller then
$\Aut(T^d_\infty)$, but one can generally still expect
$G_\infty(\beta)$ to have finite index inside a group that is
naturally associated to $f$ in many cases.  Indeed, if we let ${\bf
  G}_n$ denote the Galois group $\Gal(K(f^{-n}(t)) / K(t))$ for $t$
transcendental over $K$ and let ${\bf G}_\infty$ be the inverse limit
of the ${\bf G}_n$, we have a candidate for such a group.  The
following treats a special case of a more general question from
\cite{Conjecture}. We restrict here to the case where $f$ is a
polynomial, because the conjecture is slightly more complicated for
general rational functions.  

\begin{question}\label{qf}
  Let $f$ be a polynomial that is not conjugate to a monomial or
  Chebychev polynomial and let $\beta$ be a point that is not periodic
  or post-critical for $f$.  Is it true that we must have
  $[{\bf G}_\infty: G_\infty(\beta)] < \infty$?
\end{question}

It is easy to see that $G_\infty(\beta)$ cannot have finite index in
${\bf G}_\infty$ when $\beta$ is periodic since the number of
irreducible factors of $f^n(x) - \beta$ goes to infinity as $n$ goes
to infinity.  Likewise, when $\beta$ is post-critical, one cannot form
the tree of iterated inverse images of $\beta$ in the usual way.
Thus, this question asks if we must have
$[{\bf G}_\infty: G_\infty(\beta)] < \infty$ except in the case of the
obvious exceptions.  Recently, Benedetto and Jones \cite{BJ25} have shown that
there are PCF quadratic rational functions and $\beta$ that are
neither periodic nor post-critical such that ${\bf G}_\infty:
[G_\infty(\beta)] = \infty$, so Question \ref{qf} does not always have
a positive answer.  

In this paper we study the family of polynomials $f(x)=x^d+c$ for
$c\in K$, which up to change of variables represents all polynomials
with precisely one (finite) critical point.  If $K$ contains a
primitive $d$-th root of unity $\xi_d$, then it is easy to see that
${\bf G}_n$ is isomorphic to $[C_d]^n$, so that ${\bf G}_\infty$ is
isomorphic to the infinite iterated wreath product $[C_d]^\infty$. For
more general $K$, we see that ${\bf G}_n$ will contain $[C_d]^n$ as a
subgroup of index $[K(\xi_d): K]$ where $\xi_d$ is a primitive $d$-th
root of unity.  Iterated Galois groups of polynomials of the form $x^d
+ c$ hav been studied extensively, see \cite{RafeM, RafeD, RafeK}, for
example.

Before stating our main results, we set some notation. Throughout
this paper, unless otherwise indicated, $K$ will refer to a number
field. We say $\beta\in \Kbar$ is periodic for $f$ if
$f^n(\beta)=\beta$ for some $n\geq 1$, and we say $\beta$ is
preperiodic for $f$ if $f^m(\beta)$ is periodic for some $m \ge
0$. We say that $\beta$ is {\bf strictly preperiodic} for $f$ it is
preperiodic for $f$ but not periodic for $f$.  A rational function $f$
is said to be post-critically finite, or {\bf PCF},  if all of its critical points are
preperiodic under $f$. 

Our main result is a finite index statement for iterated Galois groups
of strictly preperiodic points for polynomials of the form
$f(x) = x^q + c$ where $q$ is a prime power and 0 is not preperiodic
under $f$ (so that $f$ is not PCF).  

\begin{thm}\label{p-theorem}
  Let $K$ be a number field.  Let $f(x) = x^q + c \in K[x]$, where $c$
  is an algebraic integer, $q = p^r$ is a power of a prime number
  $p$, and 0 is not preperiodic for $f$. Let $\beta\in \Kbar$ be strictly
  preperiodic for $f$.  Then
  \[ [{\bf G}_\infty: G_\infty(\beta)] < \infty. \]
\end{thm}

The condition on the degree of $f$ ensures that the pair $(f,\beta)$
is eventually stable (see Definition \ref{event}), a necessary
condition that is difficult to verify in general but is known to hold
in this case via work of \cite{RafeAlon} (see Theorem \ref{JL-thm}).
The proof can be made to work for more general non-PCF $x^d + c$ under
the assumption of eventual stability, but the proof is a bit more
complicated.  The condition that $\beta$ be strictly preperiodic is
what allows us to apply Proposition \ref{main-d} in place of the
diophantine conjectures such as Vojta's conjeture and the
$abc$-conjecture that are often used to prove conditional finite index
results.

We can also prove a disjointness result for fields generated by
inverse images of distinct preperiodic points.

\begin{thm}\label{disjoint-theorem} Let $K$ be a number field.  Let $f(x) = x^q + c \in K[x]$ where $c$
  is an algebraic integer $q = p^r$ is a power of a prime number $p$,
  and 0 is not preperiodic for $f$.  Let $\beta_1, \dots, \beta_t \in \Kbar$ be
  strictly preperiodic under $f$ and suppose that there are no
  distinct $j, k$ with the property that $f^m(\beta_j) = \beta_k$ for
  some $m >0$.  For each $j$, let $M_j$ denote $K_\infty(\beta_j)$.
  Then for each $j=1,\dots, t$, we have 
\[ \left[M_j \cap \left(\prod_{k \ne j} M_k\right): K\right] < \infty .\]
\end{thm}

Both Theorems \ref{p-theorem} and \ref{disjoint-theorem} hold under
slightly weaker conditions; see Theorems \ref{p2} and \ref{d2}.  

Theorem \ref{disjoint-theorem} also has a natural interpretation as a finite
index result across pre-image trees of several points (see Section
\ref{multitree}). This allows us to state a theorem about iterated
Galois groups of periodic points as well.  The statement here is
slightly more complicated, though it says essentially the same thing
as Theorem \ref{p-theorem},
namely that the iterated Galois group has finite index in a certain
``largest possible group''; see Remark \ref{per}.

The technique of the proof is very similar to that of \cite{BT2, Uni1}.
The two main differences are that we already have eventual stability
by work of \cite{RafeAlon} and that we may use Proposition \ref{main-d} in
place of more complex diophantine arguments in order to produce
primitive divisors. Proposition \ref{main-d} produces a slightly
weaker condition on our primitive prime divisors that requires some
small changes in the Galois theoretic arguments of Section \ref{Galois
  section}.  

An outline of the paper is as follows.  In Section \ref{wreath}, we
introduce some background on wreath products and irreducibility of
iterates of polynomials.  In Section \ref{prim}, we prove our main
diophantine result, Proposition \ref{main-d}, which guarantees the
existence of primes $\fp$ such that $v_\fp(f^n(0) - \beta)$ is
positive and prime to $p$ for all sufficiently large $n$; we note that
it is crucial here that $\beta$ be strictly preperiodic.  Following
that, in Section \ref{Galois section}, we prove introduce Condition R
(adapted from \cite{BT2}), and prove results showing that
$\Gal(K_n(\beta) / K_{n-1}(\beta))$ is large as possible when this
condition are satisfied at $\beta$ for $f$ and $n$.  Then, in Section
\ref{main proof}, we combine Proposition \ref{main-d} with the results
of Section \ref{Galois section} to prove Theorem \ref{p2} and Theorem
\ref{d2}, are slight generalizations of Theorem \ref{p-theorem} and
Theorem \ref{disjoint-theorem}.  Finally, in Section \ref{multitree},
we introduce the multitree associated to our points, which allows us
to phrase a finite index result for several points at once.

\medskip

\noindent {\bf Acknowledgments.}  We would like to thank Ophelia
Adams, Philipp Habegger, Liang-Chung Hsia, Rafe Jones, Alina Ostafe, and
Harry Schmidt for helpful conversations.

\section{Preliminaries}\label{wreath}

We will gather a few basic results about wreath products and
irreducibility of polynomials.

\subsection{Wreath products}\label{wreath2}
Let $G$ be a permutation group acting on a set $X$, 
and let $H$ be any group. Let $H^X$ be the group of functions from $X$ to $H$ 
with multiplication defined pointwise, or equivalently the direct product of $|X|$ copies of $H$. 
The {\em wreath product} of $G$ by $H$ is the semidirect 
product $H^X\rtimes G$, where $G$ acts on $H^X$ 
by permuting coordinates: for $f\in H^X$ and $g\in G$ we have \[f^g(x)=f(g^{-1}x)\]
for each $x\in X$. We will use the notation $G[H]$ for the wreath product, suppressing the set $X$ 
in the notation. (Another common convention is $H\wr G$ or $H\wr_X G$ if we wish to call attention to $X$.)

Fix an integer $d\geq 2$. For $n\geq 1$, let $T^d_n$ be the complete rooted $d$-ary tree of level $n$. 
It is easy to see that $\Aut(T^d_1)\cong S_d$, and standard to show that $\Aut(T^d_n)$ 
satisfies the recursive formula
\[\Aut(T^d_n)\cong \Aut(T^d_{n-1})[S_d].\]
Therefore we may think of $\Aut(T^d_n)$ as the ``$n$th iterated wreath product" of $S_d$, which we will 
denote $[S_d]^n$. In general, for $f\in K[x]$ of degree $d$ and $\beta\in K$, the Galois group $G_n(\beta)=\Gal(K_n(\beta)/K)$ 
embeds into $[S_d]^n$ via the faithful action of $G_n(\beta)$ on the $n$th level of the tree of preimages of $\beta$ 
(see for example~\cite{OdoniIterates} or~\cite[Section 2]{BT2}). 

Assume now that $f(x):=x^d+c\in K[x]$, where $K$ is a field of characteristic $0$ that contains the $d$th roots of unity. 
For $\beta\in K$ such that $\beta-c$ is not a $d$th power in $K$, 
we have $K_1(\beta)=K((\beta-c)^{1/d})$ and $G_1(\beta)\cong C_d$. For any $n\geq 2$, 
the extension $K_{n}(\beta)$ is a Kummer extension attained by adjoining to $K_{n-1}(\beta)$ the 
$d$th roots of $z-c$ where $z$ ranges over the roots of $f^{n-1}(x)=\beta$. Thus we have
\begin{equation*}
\Gal(K_n(\beta)/K_{n-1}(\beta))\subseteq \prod_{f^{n-1}(z)=\beta} \Gal(K_{n-1}(\beta)((z-c)^{1/d})/K_{n-1}(\beta))\subseteq C_q^{q^{n-1}}.
\end{equation*}
This is clear if $f^{n-1}(x)-\beta$ has distinct roots in $\Kbar$. 
If $f^{n-1}(x)-\beta$ has repeated roots, then $\Gal(K_n(\beta)/K_{n-1}(\beta))$ sits inside a direct product of a 
smaller number of copies of $C_d$, so the stated containments still hold. 

Considering the Galois tower
\[ K_n(\beta) \supseteq K_{n-1}(\beta)\supseteq K\]
we see that 
\[G_n(\beta)\subseteq \Gal(K_n(\beta)/K_{n-1}(\beta)) \rtimes G_{n-1}(\beta) \cong G_{n-1}(\beta)[C_d],\]
where the implied permutation action of $G_{n-1}(\beta)$ is on the set of roots of $f^{n-1}(x)-\beta$. By induction,
$G_n(\beta)$ embeds into $[C_d]^n$, the $n$th iterated wreath product of $C_d$. 
Observe that $[C_d]^n$ sits as a subgroup of $\Aut(T^d_n)\cong[S_d]^n$ via the obvious action on the tree. 
Taking inverse limits, $G_\infty(\beta)$ embeds into $[C_d]^\infty$, which sits as a subgroup of $\Aut(T_\infty)$. 

We summarize our basic strategy for proving that $G_\infty(\beta)$ has finite or infinite index in $[C_d]^\infty$ 
as Proposition~\ref{indexprop}.

\begin{prop}\label{indexprop}
Let $f=x^d+c\in K[x]$. Then $[{\bf G}_{\infty}: G_\infty(\beta)]<\infty$ if and only if $\Gal(K_n(\beta)/K_{n-1}(\beta))\cong C_q^{q^{n-1}}$ 
for all sufficiently large $n$.
\end{prop}
\begin{proof}
 Since $[K(\xi_d):K]$ is finite, it suffices to prove this when
 $K$ contains $\xi_d$.  Thus, we may assume that ${\bf G}_\infty$ is
 $[C_d]^\infty$.  
Consider the projection map $\pi_n:[C_d]^\infty\to[C_d]^n$. The restriction of $\pi_n$ maps $G_\infty(\beta)$ to 
$G_n(\beta)$. By basic group theory, we have
\[
[[C_d]^\infty:G_\infty(\beta)]\geq [[C_d]^n:G_n(\beta)].
\]
Therefore if $\Gal(K_n(\beta)/K_{n-1}(\beta))$ is a proper subgroup of $C_q^{q^{n-1}}$ for infinitely many $n$, then 
$[[C_d]^n:G_n(\beta)]$ is unbounded as $n\to\infty$, and $[[C_d]^\infty:G_\infty(\beta)]=\infty$.

Conversely, by appealing to the profinite structure of $[C_d]^\infty$ we see that distinct cosets of $G_\infty(\beta)$ in $[C_d]^\infty$ must 
project to distinct cosets in $[C_d]^n$ under $\pi_n$ for some $n$. If there exists $N$ such that $\Gal(K_n(\beta)/K_{n-1}(\beta))\cong C_q^{q^{n-1}}$ for all $n > N$, then by induction, 
\[[[C_d]^n:G_n(\beta)]\leq [[C_d]^N:G_N(\beta)]\] for all $n$. Thus $[[C_d]^\infty:G_\infty(\beta)]\leq [[C_d]^N:G_N(\beta)]$ as well.
\end{proof}

\subsection{Capelli's lemma and eventual stability}

We will use Capelli's lemma throughout this paper.  The lemma is
standard (see \cite{OdoniWreathProducts} or \cite[Lemma 4.1]{BT2}, for
example.  We state it here without proof.

\begin{lem}[Capelli's Lemma] \label{cap} 
Let $K$ be any field and let $f,g\in K[x]$. Suppose $\alpha\in\Kbar$ is any root of $f$. 
Then $f(g(x))$ is irreducible over $K$ if and only if both $f(x)$ is irreducible over $K$ and $g(x)-\alpha$ 
is irreducible over $K(\alpha)$.
\end{lem}

\begin{defn}\label{event}
  Let $K$ be a number field, let $f$ be a rational function, and let
  $\beta \in \bP^1(K)$.  We say that the pair $(f,\beta)$ is
  \textbf{eventually stable over $K$} if there is a constant $C$ such for any $n$, the
  the number of $\Gal({\overline K}/K)$-orbits of points in
  $f^{-n}(\beta)$ is less than $C$.  (Note that $C$ depends on $K$,
  $f$, and $\beta$ in general.)
\end{defn}

When $f$ is a polynomial and $\beta$ is not the point at infinity this
is equivalent to saying that the number of irreducible factors of
$f^n(x) - \beta$ over $K[x]$ is bounded independently of $n$.

The following is a simple application of Capelli's lemma (see
\cite[Proposition 4.1]{BT2}).

\begin{prop}\label{from-BT2}
Let $K$ be a number field, let $f \in K[x]$, and let $\beta \in K$.
If the pair $(f,\beta)$ is eventually stable over $K$, there exists some
$N\ge 0$ such that for every element of $\alpha \in f^{-N}(\beta)$, the
polynomial $f^n(x) - \alpha$ is irreducible over $K(\alpha)$ for
all $n \ge 0$.  
\end{prop}

\section{Primitive prime divisors}\label{prim}

Proposition \ref{main-d} is the main diophantine tool used in this
paper.  The proof is similar to that of \cite[Proposition 12]{BIJJ}.
It provides us with primitive prime divisors of $f^n(0)$, that is
prime divisors $\fp$ of $f^n(0)$ that are not divisors of $f^m(0)$ for any $m <
n$.  In general, a prime $\fp$ in a number field $K$ is said to be
primitive prime divisor of $a_n$ for $a_n$ an element of sequence
$(a_i)$ of elements of $K$ if $v_\fp(a_n) > 0$ and $v_\fp(a_m) \leq 0$
for all $m < n$.

We say that a polynomial
\[
f(x) = a_dx^d + a_{d-1}x^{d-1}+\dots + a_1x + a_0,
\]
has {\bf good reduction} at $\fp$ if $v_\fp(a_d)=0$ and
$v_\fp(a_i)\geq 0$ for $0\leq i\leq d-1$. See \cite{MortonSilverman}
or \cite[Theorem 2.15]{SilvermanADS} for a more careful definition
that also applies to rational functions. Clearly, any $f$ has good
reduction at all but finitely many $\fp$. The idea behind the
definition as used here is that if $f$ has good reduction at $\fp$,
then $f$ commutes with the reduction mod $\fp$ map
$r_\fp: \P^1(\Kbar)\to \P^1(\overline{k}_\fp)$. This is clear
for polynomials (see \cite[Theorem 2.18]{SilvermanADS} for a proof for
rational functions). We say that $f$ has {\bf good separable
  reduction} at $\fp$ if the reduced map
$\bar{f}:\P^1(\overline{k}_\fp)\to\P^1(\overline{k}_\fp)$ is
separable.

\begin{prop}\label{main-d}
  Let $d > 1$ be an integer, let $K$ be a number field, let
  $f$ be a polynomial of degree greater than 1, and let
  $S$ be a finite set of primes of $K$.  Let
  $\alpha \in \bP^1(K)$ be a point that is not preperiodic for
  $f$.  Let $\beta \in K$ be a point that is strictly preperiodic for
  $f$ but not post-critical for $f$.  Then for all
  sufficiently large $n$, there is a prime $\fp \notin S$ such that
  \begin{enumerate}
    \item[(1)] $v_\fp(f^n(\alpha) - \beta)$ is positive and not
      divisible by $d$; and
      \item[(2)] $v_\fp(f^m(\alpha) - \beta) = 0$ for all $m < n$.
  \end{enumerate}
  
\end{prop}

\begin{proof}
  Since $\beta$ is strictly preperiodic, the set $\{ f(\beta), \dots,
  f^k(\beta), \dots \}$ is finite and does not include $\beta$ as an
  element.  Thus, there are at most finitely many primes $\fp$ of $K$
  such that $f^k(\beta) \equiv \beta \pmod {\fp}$ for some $k >
  0$, so after expanding $S$ to a larger finite set, we may assume
  that it contains all such primes $\fp$.  Likewise, also possibly
  after expanding $S$, we may assume that the ring $\fo_S$ of
  $S$-integers in $K$ (that is elements of $K$ that are integers at
  every prime outside of $S$) is a principal ideal domain.  We may
  also assume that $S$ contains all primes of bad reduction for $f$
  and all primes $\fp$ such that $v_\fp(\alpha) < 0$ or $v_\fp(\beta)
  < 0$.  

  Now, if $v_\fp(f^n(\alpha) - \beta)$ is positive  for $\fp
  \notin S$, then we cannot have $v_\fp(f^m(\alpha) - \beta) >
  0$ for any $m < n$ since if we did we would have $f^{n-m}(\beta)
  \equiv \beta \pmod{\fp}$, which is impossible since $\fp \notin
  S$. Similarly, we cannot have $v_\fp(f^m(\alpha) - \beta) <  0$ for
  any $m$ since $v_\fp(f^m(\alpha))$ and $v_\fp(\beta)$ are both always
  non-negative since $f$ has good reduction at $\fp$ and $\alpha$ and
  $\beta$ are both integers at $\fp$.  Thus, condition (2) above will
  be met whenever (1) is, so it suffices to show that there is an
  $n_0$ such that for all $n \geq n_0$, there is an $\fp \notin S$
  such that $v_\fp(f^n(\alpha) - \beta)$ is positive and not
  divisible by $d$.

  The ring of $\fo_S^*$ of $S$-units is a finitely generated group, so
  $\fo_S^* / (\fo_S^*)^d$ is a finite group.  Let
  $\gamma_1(\fo_S^*)^d, \dots, \gamma_N(\fo_S^*)^d$ be
  the set of cosets of $(\fo_S^*)^d$ in $\fo_S^*$.  Since
  $\beta$ is not post-critical, $f^3(x) - \beta$ has at least 8 and
  thus more than 4 roots.  It follows from Riemann-Hurwitz that for
  each $i$, the curve $C_i$ given by $y^d = \gamma_i (f^3(x) - \beta)$
  has genus greater than 1. By Faltings' theorem, this means that each
  $C_i$ has finitely many rational points.  Thus, since $\alpha$ is
  not preperiodic there is an $n_0$ such that for all $n \geq n_0$
  there is no $y \in K$ such that
  $y^d = \gamma_i (f^3(f^{n-3}(\alpha)) - \beta)$.

  Now, let $n \geq n_0$ and suppose that for every prime $\fp \in S$,
  we have that $v_\fp(f^n(\alpha) - \beta)$ is either 0 or divisible
  by $d$.  Then the $\fo_S$ ideal generated by $f^n(\alpha) - \beta$
  is the $d$-th power of the $\fo_S$-ideal $I$.  Let $z$ be a generator
  for $I$ as an $\fo_S$ ideal (such a $z$ exists since $\fo_S$ is a
  principal ideal domain).  Then we have
  $z^d = u (f^n(\alpha) - \beta)$ for a unit $u \in \fo_s^*$.  We may
  write $u = \gamma_i w^d$ for one of our coset representatives
  $\gamma_i$ and some unit $w \in \fo_S^*$.  Let $y = z/w$.  Then we
  have $y^d = \gamma_i (f^3(f^{n-3}(\alpha))) - \beta)$ with $y \in K$,
  a contradiction.

 Thus, for every $n \geq n_0$, there is an $\fp \notin S$ such that
 $v_\fp(f^n(\alpha) - \beta)$ is positive and not divisible by $d$.

    


\end{proof}

\section{Ramification and Galois theory}\label{Galois section}
Throughout this section,  $f(x)$ will denote a polynomial of the form
$f(x) = x^q + c$ where $c \in K$ for $K$ a number field, the critical points 0 is not
preperiodic, and $q = p^r$ is a power of a prime $p$.  

In this section we define
{\bf Condition R} and {\bf Condition U} in terms of primes dividing
certain elements of $K$ related to the forward orbits of 0. In
Proposition~\ref{necessary} and~\ref{main Galois} we show that these
conditions control ramification in the extensions
$K(\beta) \subseteq K_n(\beta)$, with consequences for the Galois
theory of these extensions.  We begin with the following standard
lemma from Galois theory (see also \cite[Lemma~6.1]{BT2}).  

\begin{lem}\label{Galois}
Let $F_1,\dots,F_n$ and $M$ be fields all contained in some larger field. Assume that $F_1,\dots,F_n$ are finite extensions of $M$.
\begin{enumerate}[(i)]
\item If $F_1$ is Galois over $M$ and $F_1\cap F_2=M$, then $F_1 F_2$ is Galois over $F_2$ and
$\Gal(F_1 F_2/F_2)\cong\Gal(F_1/M)$.
\item If $F_1,\dots,F_n$ are Galois over $M$ with
  $F_i\cap\prod_{j\neq i}F_j = M$ for each $i$, then $\Gal(\Pi_{i=1}^n
  F_i/M)\cong\prod_{i=1}^n\Gal(F_i/M)$. 
\end{enumerate}
\end{lem}

We prove another slightly more technical lemma that we will use
throughout the rest of this paper.
\begin{lem}\label{Galois2}
  Let $M$ be a finite extension of a number field $A$.
  \begin{enumerate}[(i)]
\item Let $F_1$ and $F_2$ be finite extensions of $M$.  Suppose that
  $F_1$ is Galois over $M$ and that there is a prime $\fp$ of $A$ such
  that $\fp$ does not ramify in $F_2$ but does ramify in any
  nontrivial extension of $M$ contained in $F_1$.  Then we have
  $\Gal(F_1 F_2/F_2) \cong \Gal(F_1/M)$ and furthermore $\fp$
  ramifies in any nontrivial extension of $F_2$ contained in
  $F_1 F_2$.
\item Let $F_1,\dots,F_n$ be number fields that are all Galois
  over $M$.  Suppose that for each
  $F_i$, there is a prime $\fp_i$ of $A$ such that $\fp_i$ does not ramify in
  $F_j$ for $i \not= j$ but does ramify in any nontrivial extension of
  $M$ contained in $F_i$.  Then
  $\Gal(\Pi_{i=1}^n F_i/M)\cong\prod_{i=1}^n\Gal(F_i/M)$ and
  furthermore any nontrivial extension of $M$ that contained in
  $\Pi_{i=1}^n F_i$ must ramify over some $\fp_i$.
  \end{enumerate}
\end{lem}
\begin{proof}
For (i), we note that we must have $F_1 \cap F_2 = M$ since $F_1 \cap
F_2$ is unramified over $\fp$ and contained in $F_1$.   Then by Lemma
\ref{Galois}, we have $\Gal(F_1 F_2/F_2) \cong \Gal(F_1/M)$ .  Thus,
every extension $E$ of $F_2$ contained in $F_1 F_2$ has the form $E
=M' F_2$ for some extension $M'$ of $M$ contained in $F_1$.   If $E$ is
unramified over $\fp$, then we must have $M'=M$, by assumption, so $E$
must equal $F_2$ as desired.

To prove (ii), note first that $F_i\cap\prod_{j\neq i}F_j = M$ for
each $i$ (since $\prod_{j\neq i}F_j$ is unramified over $\fp_i$), so
$\Gal(\Pi_{i=1}^n F_i/M)\cong\prod_{i=1}^n\Gal(F_i/M)$, by Lemma
\ref{Galois}.  By part (i), every nontrivial extension of $\Pi_{j \not= i} F_j$
contained in $\Pi_{k=1}^n F_k$ ramifies over $\fp_i$.  Thus, for each
$i$, the group $G_i = \Gal(\Pi_{k=1}^n F_k/\Pi_{j \not= i} F_j)$ is
generated by inertia groups of the form $I(\fm_i / \fq_i)$ where
$\fq_i$ is a prime in $\Pi_{j \not= i} F_j$ lying over $\fp_i$. Note
that each such $I(\fm_i / \fq_i)$ is contained in
$I(\fm_i / \fq_i \cap M)$, because $\Pi_{j \not= i} F_j$ is unramified
over $\fp_i$.  Since the $G_i$ generate $\Gal(\Pi_{i=1}^n F_i/M)$,
this means that $\Gal(\Pi_{i=1}^n F_i/M)$ is generated by inertia
groups over primes in $M$ lying over the $\fp_i$.  Thus, there can be no
nontrivial extension of $M$ contained in $\Pi_{i=1}^n F_i$
that is unramified over all $\fp_i$.
 \end{proof}

We now define Conditions R and U.  These are very similar to the
definitions of Conditions R and U from \cite{BT2}.  

\begin{defn}
  Let $\beta \in \Kbar$.  We say that a
  prime $\fp$ of $K(\beta)$ satisfies {\bf Condition R} at
  $\beta$ for $f$ and $n$ if the following hold:
\begin{enumerate}[(a)]
\item $f$ has good separable reduction at $\fp$;
\item $v_\fp(f^i(0) - \beta) = 0$ for all $0 \leq i < n$;
\item $v_\fp(f^n(0) - \beta)$ is positive and prime to $p$;
\item $v_\fp(\beta) = 0$.
\end{enumerate}
\end{defn}

\begin{defn}
  Let $\beta \in \Kbar$.  We say
  that a prime $\fp$ of $K(\beta)$ satisfies {\bf Condition U} at
  $\beta$ for $f$ and $n$ if the following hold:
\begin{enumerate}[(a)]
\item $f$ has good separable reduction at $\fp$;
\item $v_\fp(f^i(0) - \beta) = 0$ for all $0 \leq i \leq n$;
\item $v_\fp(\beta)=0$.
\end{enumerate}
\end{defn}

\begin{rem}
Note that if a prime $\fp$ of $K(\beta)$ satisfies Condition R at
  $\beta$ for $f$ and $n$, then it satisfies Condition U at $\beta$
  for $f$ and $n-1$.   
\end{rem}

\begin{prop}\label{necessary}
Let $\beta\in\Kbar$. Let $\p$ be a prime of $K(\beta)$ that satisfies
Condition U at $\beta$ for $f$ and $n$. 
Then $\fp$ is unramified in $K_n(\beta)$.
\end{prop}
\begin{proof}
This is the content of \cite[Proposition 3.1]{BridyTucker}. The proof in \cite{BridyTucker} is 
stated for $\beta\in K$, but works exactly the same if we allow $\beta\in\Kbar$ and replace $K$ with $K(\beta)$.
\end{proof}

The following result is similar to \cite[Proposition~6.5]{BT2}.

\begin{prop}\label{main Galois}
Let $\beta\in\Kbar$. Suppose that $\fp$ is a prime of $K(\beta)$ that 
satisfies Condition R at $\beta$ for $n > 1$ and that $f^n(x)
- \beta$ is irreducible over $K(\beta)$. Then 
\[ \Gal( K_{n}(\beta) / K_{n-1}(\beta)) \cong C_q^{q^{n-1}}.\]
Furthermore, $\fp$
does not ramify in $K_{n-1}(\beta)$ and does ramify in any field $E$ such that
$K_{n-1}(\beta) \subsetneq E \subseteq K_n(\beta)$. Thus, we have
\begin{equation} \label{disj} \Gal( M \cdot K_{n}(\beta) / M \cdot K_{n-1}(\beta)) \cong
  C_q^{q^{n-1}}
\end{equation}
for any field $M$ containing $K(\beta)$ that does not ramify over
$\fp$.  
\end{prop}
\begin{proof}
Note that since $K_1$ contains a primitive $q$-th root of unity, and
$n > 1$, the field $K_{n-1}$ contains a primitive $q$-th root of
unity. 
Recall that Condition R at $\beta$ for $n$ implies Condition U at $\beta$ for $n-1$. 
By Proposition~\ref{necessary}, $\fp$ does not ramify in $K_{n-1}(\beta)$.

Let $\bar{z}$ denote the image of $z\in \P^1(\Kbar)$ under the
reduction mod $\fp$ map, which is well defined as long as
$v_\fp(z)\geq 0$. Consider the map
$\bar{f}:\P^1(\overline{k}_\fp)\to\P^1(\overline{k}_\fp)$ that comes
from reducing $f$ at $\fp$, and recall that Condition R assumes that
$f$ has good reduction at $\fp$.  We let ${\bar f}$ denote the
reduction of $f$ modulo $\fp$ and let ${\bar \beta}$ denote the
reduction of $\beta$ modulo $\fp$ as before.  Since 0 is the only
critical point of ${\bar f}$, it follows from (b) of Condition R that
there are no critical points of ${\bar f}^{n-1}$ in
${\bar f}^{-(n-1)}({\bar \beta})$.  By (c) of Condition R, we see that
$0\in\bar{f}^{-n}(\bar \beta)$, and that $0$ is totally ramified over
$\bar{f}(0)=\bar{c}$ (in the sense of $\bar{f}$ as a morphism from
$\P^1(\overline{k}_\fp)$ to itself).  So $\bar{f}^n(x)-\bar{\beta}$ has $0$ as a
root of multiplicity $q$, and has no other repeated roots.

Now let $z_1,\dots,z_{q^{n-1}}$ be roots of $f^{n-1}(x)-\beta$ and
$M_i$ be the splitting field of $f(x)-z_i$ over $K_{n-1}(\beta)$.
By definition, \[f^n(x)-\beta = \prod_{i=1}^{q^{n-1}} (f(x)-z_i)\]
over $K_{n-1}(\beta)$. Let $\fq$ be a prime of $K_{n-1}(\beta)$
lying over $\fp$. Then $\fq$ does not ramify over $\fp$, so
\[v_\fq (f^n(0)-\beta) = v_\fp (f^n(0)-\beta) >0.\]
Also, $z_i$'s are different modulo $\fq$, since if
$z_i\equiv z_j \pmod \fq$ for some $i\ne j$,
then $\bar z_i = \bar z_j$ which contradicts that
$\bar f^{n-1}(x)-\bar \beta$ has no repeated roots.
Therefore, we may assume
$v_\fq(f(0)-z_1)=v_\fq(f^n(0)-\beta)$ and $v_\fq(f(0)-z_i)=0$
for all $i\ne 1$.

Since $v_\fq (f(0)-z_1) = v_\fq (f^n(0)-\beta)=v_\fp(f^n(0)-\beta)$
is prime to $p$, we have $\Gal(M_1/K_{n-1}(\beta)) \cong C_q$.
On the other hand, $\fq$ does not ramify in $M_j$ for any $j\ne 1$,
because $v_\fq (z_j-f(0))=0$. It follows that
\[M_1 \cap \left(\prod_{j\ne 1} M_j\right) = K_{n-1}(\beta).\]

As $f^n(x)-\beta$ is irreducible over $K(\beta)$, it follows from the
Capelli's lemma that $f^{n-1}(x)-\beta$ is irreducible over $K(\beta)$
as well. Therefore all of the $z_i$ are Galois-conjugate.  That is,
for any $z_j\neq z_1$, there exists $\sigma\in G_{n-1}(\beta)$ such
that $\sigma (z_1)=z_j$. Applying $\sigma$ to $\fq$, we obtain a prime
$\sigma(\fq)$ of $K_{n-1}(\beta)$ that ramifies in $M_j$ with
ramification index $q$ and does not ramify in $M_k$ for any $k\neq j$.
Repeating the same argument as above, it follows that
$\Gal(M_j/K_{n-1}(\beta))\cong C_q$. Since for each $j$, there is a
prime of $K_{n-1}(\beta))$ that ramifies completely in $M_j$ but not
in $M_k$ for any $j \ne k$, so by part (ii) of Lemma \ref{Galois2} (with $A$
and $M$ taken to be $K_{n-1}(\beta)$ and the $M_i$ taken to be the
$F_i$), we have
$\Gal(K_n(\beta)/K_{n-1}(\beta))\cong C_q^{q^{n-1}}$.  Likewise, by
part (ii) of Lemma \ref{Galois2}, every nontrivial extension of $K_{n-1}(\beta)$
contained in $K_n(\beta)$ must ramify over some $\sigma(\fq)$ and thus
over $\fp$.

\end{proof}

\begin{rem}
  One might hope that Proposition \ref{main Galois} holds under a
  natural weakening of Condition R where instead of requiring that
  $v_\fp(f^i(0) - \beta) = 0$ for all $0 \leq i < n$ one only requires
  that $v_\fp(f^i(0) - \beta)$ be nonnegative and prime to $p$ for all
  $0 \leq i < n$.  However, this is not the case. Consider the case of
  $p=2$ with 
  $f(x) = x^2 -6$ and $\beta = 111$ with the prime $\fp = (3)$ in
  $\bZ$.  Then $v_\fp(f(0) - \beta) = 2$, $v_\fp(f^2(0) -\beta) = 4$,
  and $v_\fp(f^3(0 - \beta) = 3$. One can calculate that $\fp$ does
  not ramify in $K_1(\beta)$, ramifies with index 2 in $K_2(\beta)$,
  and ramifies with index 4 in $K_3(\beta)$.  Moreover for each prime
  $\fq$ in $K_2(\beta)$ lying over $\fp$, one can see that
  $K_3(\beta)$ contains a nontrivial extension of $K_2(\beta)$ that is
  unramified over $\fq$.  Thus, the condition that
  $v_\fp(f^i(0) - \beta) = 0$ for all $0 \leq i < n$ is necessary in
  Proposition \ref{main Galois}.
\end{rem}

With this notation we have the following result, which is similar to
\cite[Proposition~6.7]{BT2}.  First we need a little notation
extending our earlier notation.  Let $\ba = (\alpha_1, \dots,
\alpha_s)  \in \Kbar^{s}$ and $L=K(\alpha_1,\dots,\alpha_s)$.
We let $K_n(\ba)$ denote the compositum $K_n(\alpha_1)
\cdots K_n(\alpha_s)$.  We let $G_n(\ba)$ denote $\Gal(K_n(\ba)/L)$
and let $G_\infty(\ba)$ be the inverse limit of the $G_n(\ba)$.

\begin{prop}\label{final Galois}
  Let $\ba = (\alpha_1, \dots, \alpha_s)$ and $L$ be the same as above
  and $n>0$. Suppose there exist primes $\fq_1,\dots,\fq_s$ of $L$ such that
\begin{enumerate}[(a)]
\item $\fq_i \cap K(\alpha_i)$ satisfies Condition R at
$\alpha_i$ for $f$ and $n$;
\item $\fq_i \cap K(\alpha_j)$ satisfies Condition U at
$\alpha_j$ for $f$ and $n$ for all $j \ne i$;
\item $\fq_i \cap K$ does not ramify in $L$; and
\item $f^n(x) - \alpha_i$ is irreducible over $K(\alpha_i)$ for
  all $i = 1, \dots, s$.
\end{enumerate}
Then  $\Gal(K_{n}(\ba)/ K_{n-1}(\ba)) \cong C_q^{s q^{n-1}}$.
Furthermore, for any field $E$ with $K_{n-1}(\ba) \subsetneq E \subset
K_{n}(\ba)$, there is an $i$ such that $\fq_i \cap K$ ramifies in
$E$.   
\end{prop}

\begin{proof}
  For each $1\le i \le s$, by (a), (d), and Proposition~\ref{main
    Galois} (with $\fq_i \cap K$ playing the role of $\fp$ in the
  statement) we have
  \[\Gal(K_n(\alpha_i)/K_{n-1}(\alpha_i))\cong C_q^{q^{n-1}}.\]
  We also have that every nontrivial extension of $K_{n-1}(\alpha_i)$
  contained in $K_n(\alpha_i)$ must ramify over $\fq_i \cap K$.

  Now for each $1\leq i\leq s$, let $\fp_i = \fq_i\cap K$ and
  $L_i = K_n(\alpha_i)\cdot K_{n-1}(\ba)$. By (a), (b), (c), and
  Proposition~\ref{necessary}, the prime $\fp_i$ does not ramify in $L_j$ for
  all $j\neq i$ and also does not ramify in $K_{n-1}(\ba)$.  By part
  (i) of Lemma \ref{Galois2}, we see that every nontrivial extension of
  $K_{n-1}(\ba)$ of contained in $L_i$ must ramify over $\fp_i$.
  Thus, we may apply part (ii) of Lemma \ref{Galois2} (with the
field $M$ as $K_{n-1}(\ba)$, the field $A$ as $K$, and the fields
$F_i$ as $L_i$) to obtain

  \begin{align*}
  \Gal \left(K_n(\ba)/K_{n-1}(\ba)\right)
  & \cong \prod_{i=1}^s \Gal \left(L_i/K_{n-1}(\ba)\right) \\
  & \cong \prod_{i=1}^s \Gal \left(K_n(\alpha_i)/K_{n-1}(\alpha_i)\right)
  \cong C_q ^{sq^{n-1}}.
 \end{align*}

 It also follows from part (ii) of Lemma \ref{Galois2} that every
 nontrivial extension of $K_{n-1}(\ba)$ in $K_n(\ba)$ must ramify over
 some $\fp_i$.

\end{proof}


\section{Proof of Main Theorems}\label{main proof}


The proofs of the main theorems combine the preliminary arguments from 
throughout the paper with the following proposition, which uses height arguments to 
produce primes with certain ramification behavior in $K_n(\beta)$. 
Recall the definitions of Condition R and Condition U from Section \ref{Galois section}.
\begin{prop}\label{final prop}
  Let $K$ be a number field, $q = p^r$ be a power of a prime
  number $p$, and $f(x) = x^q + c \in K[x]$ be a
  polynomial that is not PCF. Let $\ba=(\alpha_1, \dots, \alpha_s)$ be distinct
  strictly preperiodic points of $f$ such that for each $i \ne j$
  there is no $\ell \geq 0$ such that $f^{\ell}(\alpha_i) = \alpha_j$.
  Then there is an $n_0$ such that for all $n \geq n_0$, there are
  primes $\fq_1, \dots \fq_s$ of $L=K(\alpha_1,\dots,\alpha_s)$ such that
\begin{enumerate}[(i)]
\item for each $i$, we have that $\fq_i \cap K(\alpha_i)$
satisfies Condition R at
$\alpha_i$ for $n$;
\item for each $i\ne j$, we have that $\fq_i \cap K(\alpha_j)$
satisfies Condition U at $\alpha_j$ and $f$ for all $m \geq 0$;
\item $\fq_i \cap K$ does not ramify in $L$.   
\end{enumerate}
\end{prop}
\begin{proof}
  Let $\cS$ be a set of primes in $L$ that includes all primes of $L$ of
  bad or inseparable reduction for $f$, all primes $\fq$ of $L$ such that
  $v_\fp(\alpha_i) \ne 0$ for some $i$ (note that none of the $\alpha_i$
  can equal to 0 since 0 is not preperiodic under $f$), all primes $\fq$
  of $L$ such that $\fq \cap K$ ramifies in $L$, and all primes of $L$
  such that $f^\ell(\alpha_i) \equiv \alpha_j \pmod{\fq}$ for some
  $\ell \geq 0$ and some $i \ne j$.
  By our assumptions, $\cS$ is a finite set.
  Now for each $i$, let \[\cS_i = \{\fq \cap K(\alpha_i):\fq\in \cS\}.\]
  By Proposition \ref{main-d}, for all sufficiently large $n$,
  there is a prime $\fp_i \notin \cS_i$ of $K(\alpha_i)$
  that satisfies Condition R at $\alpha_i$ for $n$.
  For each $i$, choose a prime $\fq_i$ of $L$ lying over $\fp_i$.
  Then $\fq_i$ satisfies condition (i) and (iii).
  
  Since $\fq_i \notin \cS$,
  there is no $\ell \ge 0$ such that
  $f^\ell(\alpha_j) \equiv \alpha_i \pmod {\fq_i}$; it follows that we
  cannot have $f^m(0) \equiv \alpha_j \pmod{\fq_i \cap K(\alpha_j)}$
  for any $m\geq 0$, since otherwise we would have
  \[f^{n-m}(\alpha_j) \equiv f^n(0) \equiv \alpha_i \pmod{\fq_i}\]
  when $n\geq m$ or
  \[f^{m-n}(\alpha_i) \equiv f^m(0) \equiv \alpha_j \pmod{\fq_i}\]
  when $m \geq n$.
  Thus $\fq_i$ also satisfies condition (ii).
\end{proof}

We will use the following theorem of Jones and Levy \cite{RafeAlon}.
It is a special case of their Theorem 1.3.  

\begin{thm}\label{JL-thm}
Let $K$ be a number field and let $f(x) = x^q + c\in K[x]$ where $q = p^r$ for
a prime $p$ and $v_\fp(c)\ge 0$ for some prime $\fp$ lying over $p$.  Then for any $\beta \in K$ that is
not periodic under $f$, the pair $(f,\beta)$ is eventually stable over $K$.  
\end{thm}

The following is a slight generalization of Theorem \ref{p-theorem}.  
\begin{thm}\label{p2}
  Let $K$ be a number field.  Let $q = p^{r}$ ($r\ge 1$) be a power of
  the prime number $p$, let $f(x) = x^q + c \in K[x]$, where $v_\fp(c)\geq 0$ for some prime $\fp$ in $K$ lying over $p$, and let
  $\beta \in \Kbar$ be strictly preperiodic for $f$. Suppose that 0 is not
  preperiodic under $f$.  Then we have
  $[
  {\bf G}_\infty:G_\infty(\beta)] < \infty$.
\end{thm}

\begin{proof}
  By extending $K$, we may assume that $\beta\in K$ and $K$ has a primitive $q$-th
  root of unity. Then by Theorem
  \ref{JL-thm}, the pair $(f,\beta)$ is eventually stable over $K$.
  From Proposition \ref{from-BT2}, there is some $N\ge 0$ such that for all
  $\alpha \in f^{-N}(\beta)$ and for all $n\geq 1$, the polynomial
  $f^n(x)-\alpha$ is irreducible over $K(\alpha)$, which implies
  condition (d) in Proposition \ref{final Galois}.
  On the other hand, applying Proposition
  \ref{final prop} to $\ba=f^{-N}(\beta)$, there is an $n_0$ such that
  for all $n\ge n_0$, there are primes $\fq_1,\dots,\fq_s$ of 
  $L=K_N(\beta)$ satisfying conditions (i), (ii), and (iii) in
  Proposition \ref{final prop}, which imply conditions (a), (b), and (c)
  in Proposition \ref{final Galois}. (Here $s=q^{N}$.)
  Therefore, we have
  \[\Gal(K_{N+n}(\beta)/K_{N+n-1}(\beta))\cong C_q^{sq^{n-1}} = C_q^{q^{N+n-1}}\]
  for all $n\ge n_0$. By Proposition~\ref{indexprop}, we are done.
\end{proof}

\begin{thm}\label{d2}
  Let $K$ be a number field.  Let $f(x) = x^q + c \in K[x]$ where
  $q = p^r$ is a power of a prime number $p$ and $v_\fp(c)\ge 0$ for
  some prime $\fp$ of $K$ lying over $p$.  Suppose that 0 is not
  preperiodic under $f$.  Let
  $\beta_1, \dots, \beta_t \in K$ be preperiodic under $f$ and suppose
  that there are no distinct $j, k$ with the property that
  $f^\ell(\beta_j) = \beta_k$ for some $\ell>0$.  For each $j$, let $M_j$
  denote $K_\infty(\beta_j)$.  Then for each $j=1,\dots, t$, we
  have 
\[ \left[M_j \cap \left(\prod_{k \ne j} M_k\right): K\right] < \infty .\]
\end{thm}

\begin{proof}
Applying Proposition \ref{from-BT2} to each $\beta_j$ and take the maximum,
we have some
$N\ge 0$ such that for all $\alpha \in B=\bigcup_{j=1}^t f^{-N}(\beta_j)$
and for all $n \geq 1$, the polynomial $f^n(x) - \alpha$ is
irreducible over $K(\alpha)$. Let $n_0$ be as
in Proposition \ref{final prop} for $\ba=B$.

Now fix a $j$. We claim that
\[  M_j \cap \prod_{k \ne j} M_k \subseteq K_{N+n_0}(\beta_j) \] so that
it has a finite index over $K$. It suffices to show that
\[ K_{N+n}(\beta_j) \cap  \prod_{k \ne j} M_k \subseteq K_{N+n_0}(\beta_j)\]
for all $n > n_0$, which is equivalent to
\begin{equation*}
  \left[K_{N+n}(\beta_j) \cdot \prod_{k \ne j} M_k :
  K_{N+n_0}(\beta_j) \cdot \prod_{k \ne j} M_k \right]
  = [K_{N+n} (\beta_j) : K_{N+n_0}(\beta_j)]
\end{equation*}   
for all $n > n_0$.  Using induction, it will suffice to show that
\begin{equation}\label{induct}
  \begin{split}
  & \left[K_{N+n}(\beta_j) \cdot \prod_{k \ne j}  M_k  : 
  K_{N+n-1}(\beta_j) \cdot \prod_{k \ne j} M_k \right] \\
   & = [K_{N+n} (\beta_j) : K_{N+n-1}(\beta_j)]
\end{split}
\end{equation} for all $n > n_0$.

Let $\alpha_i \in f^{-N}(\beta_j)$.
For each $n > n_0$, there is a prime $\fq_i$ in
$L=K(B)$ corresponding to $\alpha_i$, satisfying
(i), (ii), and (iii) of Proposition \ref{final prop}.
$\fq_i \cap K$ does not ramify in $\prod_{k\ne j} M_k$ due to (ii) and (iii)
of Proposition \ref{final prop} and Proposition \ref{necessary}.
Also, $\fq_i \cap K$ does not ramify in $K_{N+n_0}(\beta_j)$ due to (i)
and (iii) of Proposition \ref{final prop} and Proposition \ref{necessary}.
Therefore, by Proposition \ref{main Galois}, we have that
\ref{induct} holds and our proof is complete.

\end{proof}

\section{The multitree}\label{multitree}

In this section we introduce a generalization of trees, which we call 
multitrees, in order to give a pleasant interpretation of Theorem~\ref{disjoint-theorem} in 
terms of a finite index statement. For our purposes, we can simplify the presentation of multitrees 
in~\cite[Section 11]{BT2} by avoiding the use of stunted trees.

Let $f\in K(x)$ with $\deg f\geq 2$ and set $\ba = \{\alpha_1,\dots,\alpha_s\}\subseteq K$. 
Define
\[
\cM_n(\ba)= \bigcup_{i=0}^n \bigcup_{j=1}^s f^{-i}(\alpha_j)
\]
and
\[
G_n(\ba) = \Gal(K(\cM_n(\ba ))/K(\ba )).
\]
We refer to $\cM_n(\ba)$ as a \emph{multitree}. It can be pictured as the union of $s$ distinct trees of level $n$, rooted at the $\alpha_i$.

As $n\to\infty$, define the direct limit \[\cM_\infty(\ba) = \lim_{\longrightarrow} \cM_n(\ba)\]
 and the inverse limit \[G_\infty(\ba) = \lim_{\longleftarrow} G_n(\ba)\] 
 just as in the single tree case.
For each $n$, $G_n(\ba) $ acts faithfully on 
$\cM_n(\ba)$ in the usual way. So there are injections 
$G_n(\ba) \hookrightarrow\Aut(\cM_n(\ba))$, and thus an injection 
$G_\infty(\ba) \hookrightarrow\Aut(\cM_\infty(\ba))$, where an automorphism of the 
multitree must fix each root $\alpha_i$.

Suppose that the individual trees rooted at $\alpha_i$ are disjoint, and that each $\alpha_i$ is neither periodic 
nor postcritical for $f$. Then the automorphism group of the infinite multitree has the simple description
\[\Aut(\cM_\infty(\ba))\cong \Aut(T_\infty^d)^s,\]
that is, the direct product of $s$ copies of $\Aut(T_\infty^d)$. 
This group has a subgroup $\left([C_q]^\infty\right)^s$, which is the direct product of $s$ copies of the 
permutation group given 
by the infinite iterated wreath product action of $C_q$ on $T_\infty^d$. If there are 
$s$ different polynomial maps $f(x)=x^d+c_i$ that satisfy the hypotheses of Theorem~\ref{p-theorem}, 
then it is easy to see that 
$G_\infty(\ba)$ embeds into $\left([C_q]^\infty\right)^s$. Thus we may rephrase Theorem~\ref{disjoint-theorem} 
as a finite index statement.  It is most easily stated when $K$
contains a $q$-th root of unity.  

\begin{thm}\label{multitree theorem}
  Let $K$ be a number field that contains a $q$-th root of unity and
  suppose that $\alpha_i\in K$ are strictly preperiodic for $f$.
  Suppose that there are no $i \ne j$ with the property that
  $f^{\ell}(\alpha_i) = \alpha_j$ for some $\ell \geq 0$. Then
  \[[\left([C_q]^\infty\right)^s:G_\infty(\ba)]<\infty.\]
\end{thm}

\begin{proof}
The group $G_\infty(\ba)$ equals $\Gal(\prod_{i=1}^s K_\infty(f,\alpha_i)/K)$, which has 
finite index in the direct product $G_\infty(f,\alpha_1)\times\dots\times G_\infty(f,\alpha_s)$ by 
Theorem~\ref{disjoint-theorem} and basic Galois theory. This group in turn has finite index 
in $\left([C_q]^\infty\right)^s$ by applying Theorem~\ref{p-theorem} to each $G_\infty$ separately. 
\end{proof}

\begin{rem}\label{per}
  For periodic $\alpha$, the point $\alpha$ appears repeatedly as its
  own inverse image, so the natural tree for $G_\infty(\alpha)$ act on
  is the product of the rooted binary trees corresponding to the
  strictly preperiodic elements of $f^{-1}(\alpha)$, so Theorem
  \ref{multitree theorem} provides a finite index theorem for periodic
  points of $x^q + c$ as well.   
\end{rem}


\begin{thebibliography}{BDG{\etalchar{+}}21b}

\bibitem[BDG{\etalchar{+}}21a]{Uni1}
A.~Bridy, J.~R. Doyle, D.~Ghioca, L.-C. Hsia, and T.~J. Tucker, \emph{Finite
  index theorems for iterated {G}alois groups of unicritical polynomials.},
  Trans. Amer. Math. Soc \textbf{374} (2021), no.~1, 733--753.

\bibitem[BDG{\etalchar{+}}21b]{Conjecture}
\bysame, \emph{A question for iterated {G}alois groups in arithmetic dynamics},
  Canad. Math. Bull. \textbf{64} (2021), no.~2, 401--417.

\bibitem[BIJ{\etalchar{+}}17]{BIJJ}
A.~Bridy, P.~Ingram, R.~Jones, J.~Juul, A.~Levy, M.~Manes,
  S.~Rubinstein-Salzedo, and J.~H. Silverman, \emph{Finite ramification for
  preimage fields of post-critically finite morphisms}, Math. Res. Lett.
  \textbf{24} (2017), no.~6, 1633--1647.

\bibitem[BJ19]{OC}
R.~L. Benedetto and J. Juul, \emph{Odoni's conjecture for number
  fields}, Bull. Lond. Math. Soc. \textbf{51} (2019), no.~2, 237--250.

\bibitem[BJ25]{BJ25}
R.~L. Benedetto and R. Jones, \emph{Specializations of the {B}asilica
  group}, in preparation, 2025.  


\bibitem[BT18]{BridyTucker}
A.~Bridy and T.~J. Tucker, \emph{{$ABC$} implies a {Z}sigmondy principle for
  ramification}, J. Number Theory \textbf{182} (2018), 296--310.

\bibitem[BT19]{BT2}
\bysame, \emph{Finite index theorems for iterated {G}alois groups of cubic
  polynomials}, Math. Ann. \textbf{373} (2019), no.~1-2, 37--72. \MR{3968866}

\bibitem[DK22]{False}
P. Dittmann and B. Kadets, \emph{Odoni's conjecture on arboreal {G}alois
  representations is false}, Proc. Amer. Math. Soc. \textbf{150} (2022), no.~8,
  3335--3343.

\bibitem[HJM15]{RafeK}
S. Hamblen, R. Jones, and K. Madhu, \emph{The density of primes in
  orbits of {$z^d+c$}}, Int. Math. Res. Not. IMRN (2015), no.~7, 1924--1958.

\bibitem[JL17]{RafeAlon}
R.~Jones and A.~Levy, \emph{Eventually stable rational functions}, Int. J.
  Number Theory \textbf{13} (2017), no.~9, 2299--2318.

\bibitem[Jon07]{RafeM}
R. Jones, \emph{Iterated {G}alois towers, their associated martingales, and
  the {$p$}-adic {M}andelbrot set}, Compos. Math. \textbf{143} (2007), no.~5,
  1108--1126.

\bibitem[Jon08]{RafeD}
\bysame, \emph{The density of prime divisors in the arithmetic dynamics of
  quadratic polynomials}, J. Lond. Math. Soc. (2) \textbf{78} (2008), no.~2,
  523--544.

\bibitem[Jon13]{RafeS}
\bysame, \emph{Galois representations from pre-image trees: an arboreal
  survey}, Actes de la {C}onf\'erence ``{T}h\'eorie des {N}ombres et
  {A}pplications'', Publ. Math. Besan\c con Alg\`ebre Th\'eorie Nr., vol. 2013,
  Presses Univ. Franche-Comt\'e, Besan\c con, 2013, pp.~107--136.

\bibitem[Kad20]{Kadets}
B. Kadets, \emph{Large arboreal {G}alois representations}, J. Number Theory
  \textbf{210} (2020), 416--430.

\bibitem[Loo19]{NL}
N. Looper, \emph{Dynamical {G}alois groups of trinomials and {O}doni's
  conjecture}, Bull. Lond. Math. Soc. \textbf{51} (2019), no.~2, 278--292.

\bibitem[MS94]{MortonSilverman}
P.~Morton and J.~H. Silverman, \emph{Rational periodic points of rational
  functions}, Internat. Math. Res. Notices (1994), no.~2, 97--110.

\bibitem[Odo85]{OdoniIterates}
R.~W.~K. Odoni, \emph{The {G}alois theory of iterates and composites of
  polynomials}, Proc. London Math. Soc. (3) \textbf{51} (1985), no.~3,
  385--414.

\bibitem[Odo88]{OdoniWreathProducts}
\bysame, \emph{Realising wreath products of cyclic groups as {G}alois groups},
  Mathematika \textbf{35} (1988), no.~1, 101--113.

\bibitem[Ser72]{Serre}
J.-P. Serre, \emph{Propri\'et\'es galoisiennes des points d'ordre fini des
  courbes elliptiques}, Invent. Math. \textbf{15} (1972), no.~4, 259--331.

\bibitem[Sil07]{SilvermanADS}
J.~H. Silverman, \emph{The arithmetic of dynamical systems}, Graduate Texts in
  Mathematics, vol. 241, Springer, New York, 2007.

\end{thebibliography}

\newcommand{\etalchar}[1]{$^{#1}$}
\providecommand{\bysame}{\leavevmode\hbox to3em{\hrulefill}\thinspace}
\providecommand{\MR}{\relax\ifhmode\unskip\space\fi MR }
\providecommand{\MRhref}[2]{%
  \href{http://www.ams.org/mathscinet-getitem?mr=#1}{#2}
}
\providecommand{\href}[2]{#2}

\end{document}